\documentclass[reqno]{amsart}

\usepackage{amsmath,mathbbol,amssymb,mathbbol}

\usepackage[active]{srcltx}
\usepackage{graphicx,color}

\usepackage{epic}
\usepackage{pstricks}
\usepackage{amsthm}

\newtheorem{theorem}{Theorem}[section]
\newtheorem{proposition}[theorem]{Proposition}

\newtheorem{lemma}[theorem]{Lemma}

\newtheorem{corollary}[theorem]{Corollary}
\theoremstyle{definition}

\def\R{\mathbb{R}}

\def\N{\mathbb{N}}

\def \sgn{\operatorname{sgn}}

\def\trXY1-tq{Tr(X,Y, 1-\theta ,q)}

\definecolor{darkred}{rgb}{0.7,0.1,0.1}

\title[The Laplace operator on the Sierpinski gasket]{The Laplace operator on the Sierpinski gasket with Robin boundary conditions}

\author{Brigitte E. Breckner}
\address{B.~E.~Breckner, Babe\c{s}-Bolyai University, Faculty of Mathematics and Computer Science, Department of Mathematics, 400084 Cluj-Napoca, Romania}
\email{brigitte@math.ubbcluj.ro}

\author{Ralph Chill}
\address{R.~Chill, Institut f\"ur Analysis, Fachrichtung Mathematik, TU Dresden, 01062 Dresden, Germany}
\email{ralph.chill@tu-dresden.de}

\begin{document}

\date{\today}



\begin{abstract} 
We study the Laplace operator on the Sierpinski gasket with nonlinear Robin boundary conditions. We show that for certain Robin boundary conditions the Laplace operator generates a positive, order preserving, $L^\infty$-contractive semigroup which is sandwiched (in the sense of domination) between the semigroups generated by the Dirichlet-Laplace operator and the Neumann-Laplace operator. We also characterise all local semigroups which are sandwiched between these two extremal semigroups by showing that their generators are Robin-Laplace operators.
\end{abstract}

\renewcommand{\subjclassname}{\textup{2010} Mathematics Subject Classification}

\maketitle

\section{Introduction}

Laplace type operators on fractals and Dirichlet forms on fractals have been studied during the last about 20 years. For the Sierpinski gasket Kigami \cite{Ki89,Ki93} has given a definition of a Laplace operator which, on the one hand, is a natural extension of the Laplace operator on the unit interval (the Sierpinski gasket in dimension $1$) and, on the other hand, appears to be the limit of the Laplace operators on the subgraphs which come up in the construction of the Sierpinski gasket. We refer the reader to the monograph by Strichartz \cite{St06} for an exposition of the theory and to the articles \cite{FuSh92,HePeSt04,IoRoTe12,MeStTe04,Ms97,Ms98b,Ms98,StWo04} for further results. In particular, the Laplace operator on the Sierpinski gasket admits a variational definition in the sense that the Laplace operator is the operator associated with an appropriate Dirichlet form or a quadratic energy. For the Laplace operator associated with the free energy, which is the limit of a sequence of energies of appropriate graph Laplace operators, there exists even a Gau{\ss}-Green formula, an interpretation of the normal derivative on the intrinsic boundary and consequently an interpretation of associated Neumann boundary conditions. The purpose of this article is to study Laplace type operators on the Sierpinski gasket with Robin type boundary conditions, and in particular to study the influence of boundary conditions on the evolution generated by these operators. We define the Laplace operators as the subgradients of energies which need not to be quadratic, so that our study includes nonlinear Robin type boundary conditions. We show that for certain Robin type boundary conditions these subgradients generate positive, order preserving and $L^\infty$-contractive semigroups on $L^2_\mu (V)$, where $V$ is the Sierpinski gasket and $\mu$ is an appropriate Borel measure on $V$ which need not be the selfsimilar measure, that is, a multiple of the $d$-dimensional Hausdorff measure, $d$ being the Hausdorff dimension of $V$. The energies we consider are thus nonlinear Dirichlet forms in the sense introduced by B\'enilan \& Picard \cite{BePi79a} and later by Cipriano \& Grillo \cite{CiGr03}. We show in addition that the generated semigroups are sandwiched between the semigroup generated by the Laplace operator with Dirichlet boundary conditions and the semigroup generated by the free Laplace operator with Neumann boundary conditions. Finally, we characterise all local semigroups which are sandwiched between the latter two semigroups by showing that these are exactly the semigroups generated by the Laplace operator with Robin boundary conditions. This result is an analogue of corresponding result from Arendt \& Warma \cite{ArWa03b} and Chill \& Warma \cite{ChWa12} respectively for the Laplace operator and the $p$-Laplace operator on domains in $\R^N$.\\

 {\it Notations.} We denote by $\N$ the set of natural
numbers $\{0,\, 1,\, 2,\dots\}$, by $\N^*:=\N\setminus\{0\}$ the set
of positive naturals, and by $|\cdot|$ the Euclidean norm on the
spaces $\R^N$, $N\in\N^*$. The spaces $\R^N$ are endowed with the topology induced by $|\cdot|$.

\section{Preliminaries}

Throughout the paper, we denote by $V$ 
the {\it Sierpinski gasket} in $\R^{N-1}$,  where $N\geq2$ is a fixed natural number. There are two different approaches that lead to $V$, starting from given points $p_1,\dots, p_N\in\R^{N-1}$ with $|p_i-p_j|=1$ for $i\neq j$, and from the similarities 
$F_i\colon\R^{N-1}\to\R^{N-1}$, defined by
\[
F_i(x)=\frac12\,x+\frac12\,p_i, 
\]
for $i\in\{1,\dots,N\}$. While in the first approach the set $V$ appears as the unique nonempty compact subset of $\R^{N-1}$ satisfying the equality
\begin{equation}\label{defST}
V=\bigcup_{i=1}^N F_i(V),
\end{equation}
in the second one $V$ is obtained as the closure of the set $V_*:=\bigcup_{m\in\N}V_m$, where 
\begin{equation}\label{Vm}
V_0:=\{p_1,\dots,p_N\}\hbox{ and } V_m:=\bigcup_{i=1}^N F_i(V_{m-1}), \hbox{ for } m\in\N^*.
\end{equation}
In what follows $V$ is
considered to be endowed with the relative topology induced from the Euclidean
topology on $\R^{N-1}$. The set $V_0$ is called the
{\it intrinsic boundary} of the Sierpinski gasket.

For every $m\in\N^*$ denote by ${\mathfrak W}_m:=\{1,\dots, N\}^m$. For $w=(w_1,\dots,w_m)\in {\mathfrak W}_m$ put
$F_w:=F_{w_1}\circ\dots \circ F_{w_m}$. The equality (\ref{defST})
clearly yields
\begin{equation*} 
V=\bigcup_{w\in {\mathfrak W}_m}F_w(V).
\end{equation*}
This equation is the {\it level $m$ decomposition of} $V$,
and each $F_w(V)$, $w\in {\mathfrak W}_m$, is called a {\it cell of
level} $m$, or, for short, an $m$-cell. We refer to $V_0$ as the $0$-cell.

The Sobolev type spaces $H^1(V)$ and $H_0^1(V)$ on the Sierpinski gasket are obtained as subsets of the spaces $C(V)$ and $C_0(V)$, respectively, where  $C(V)$ is the space of real-valued continuous functions on $V$, and 
\[
C_0(V):=\{u\in C(V)\mid u|_{V_0}=0\}
\]
is the space of continuous functions on $V$ which vanish on the intrinsic boundary. Both spaces $C(V)$ and $C_0(V)$ are endowed with the usual supremum norm $\|\cdot\|_{\sup}$. The basic ingredient for defining the Sobolev type spaces on the Sierpinski gasket is a certain energy form which involves difference quotients and presents some analogy to the Dirichlet energy associated with the Laplace operator on domains in $\R^N$. In order to define this energy form, we 
follow both \cite[Section 1.3]{St06}, where these aspects are presented for $N\in\{2,\,3\}$, and \cite{Br13}, where  they are treated for arbitrary values $N\geq2$.  For this, consider first the sets $V_m$, $m\in\N$, defined in (\ref{Vm}). Let $m\in\N$. For $x$, $y\in V_m$  set $x\underset{m}{\sim} y$ if there 
is a cell of level $m$ containing both $x$ and $y$. Now, for functions $u$, $v\colon V_{m}\to \R$ we define the $m$-energy $W_m(u)$ by 
\begin{equation}\label{energyVm}
	W_{m}(u):= \frac12 \, \left(\frac{N+2}{N}\right)^m \sum_{\underset{x\underset{m}{\sim} y}{x,y\in V_m}}(u(x)-u(y))^2 ,
\end{equation}
and the semi-inner product $\langle u,v\rangle_m$ by
\begin{equation}\label{bilinenergyVm}
\langle u,v\rangle_m := \left(\frac{N+2}{N}\right)^m \sum_{\underset{x\underset{m}{\sim} y}{x,y\in V_m}} (u(x)-u(y))(v(x)-v(y)).
\end{equation}
Note that $W_m (u) = \frac12 \langle u,u\rangle_m$. Recall that $V_*$ is the union of the sets $V_m$, $m\in\N$. For functions $u$, $v\colon V_*\to\R$ or $u$, $v\colon V\to\R$ we denote simply by $W_m(u)$ the corresponding energy of the restrictions of $u$ to $V_m$, and we do similarly for $\langle u,v\rangle_m$. According to \cite[Corollary 3.3]{Br13} (see also \cite[pp. 14, 15]{St06}), for a function $u\colon V_*\to\R$, the sequence $\left(W_m(u)\right)_{m\in\N}$ is increasing. Thus it makes sense to define  its energy $W(u)$ by
\begin{equation}\label{defW}
W(u):=\lim_{m\to\infty} W_m (u) \in [0,\infty ].
\end{equation}
It can be shown (see the explanations on \cite[p. 19]{St06} in the case $N\in\{2,\,3\}$, respectively \cite[Theorem 4.4]{Br13} for arbitrary $N$) that functions of finite energy are H\"{o}lder continuous, hence uniformly continuous. In particular, functions of finite energy admit a unique continuous extension to $V$, and by identifying uniformly continuous functions on $V_*$ with their continuous extensions on $V$, it thus makes sense to say that $W$ is defined on the space $C(V)$. Define now
\[
H^1(V):=\{u\in C(V)\mid W(u) < \infty \} \quad \hbox{ and } \quad H_0^1(V):=H^1(V)\cap C_0(V).
\]
Using the polarization identity, it can be proved that for every $u$, $v\in H^1 (V)$ the sequence $\left(\langle u,v\rangle_m \right)_{m\in\N}$ is convergent, and that the function $\langle\cdot ,\cdot\rangle \colon H^1(V)\times H^1(V)\to\R$ given by
\begin{equation}\label{defprodscgeneral}
\langle u,v\rangle =\lim_{m\to\infty} \langle u,v\rangle_m \quad \text{for all } u,\,v\in H^1(V) ,
\end{equation}
is a semi-inner product satisfying $W(u) = \frac12 \langle u,u\rangle$ ($u\in H^1 (V)$). 

We summarise the main results of \cite[Section 1.4]{St06} (see also \cite{Br13}) concerning the basic properties of $H^1(V)$ and $H_0^1(V)$ in the following theorem.

\begin{theorem}\label{properties} 
The following assertions hold:
\begin{itemize}
\item[$1^\circ$] The spaces $H^1(V)$ and $H^1_0 (V)$ are linear, dense subspaces of $(C(V),\|\cdot\|_{\sup})$ and $(C_0 (V),\|\cdot\|_{\sup})$, respectively.

\item[$2^\circ$] For $\alpha :=\frac{\ln \frac {N+2}{N}}{2\ln 2}$ there exists a constant $c>0$ such that for every $u\in H^1(V)$ and every $x$, $y\in V$
\begin{equation}\label{Hoelder}
|u(x)-u(y)|\leq c|x-y|^\alpha\, \sqrt{W(u)} .
\end{equation}

\item[$3^\circ$] For any fixed $y\in V$
\[
 \langle u,v\rangle_{H^1} := u(y)v(y) + \langle u,v\rangle \quad (u,\, v\in H^1 (V))
\]
defines an inner product on $H^1 (V)$. 

\item[$4^\circ$] The spaces $(H^1(V),\langle\cdot,\cdot\rangle_{H^1})$ and $(H^1_0 (V),\langle\cdot,\cdot\rangle )$ are real Hilbert spaces. 

\item[$5^\circ$] The embedding of $(H^1(V),\|\cdot\|_{H^1})$ into $(C(V),\|\cdot\|_{\sup})$ is compact. 
\end{itemize}
\end{theorem}

\begin{proof}
Assertion $1^\circ$ follows by using a pointwise approximation on the sets $V_m$ and the concept of harmonic extension (see for example \cite[Theorem 1.4.4]{St06} in the case $N\in\{2,\,3\}$, and \cite[Theorem 4.6]{Br13} for arbitrary $N$). For the H\"older continuity (assertion $2^\circ$) we refer to \cite[Theorem 4.4]{Br13}. Assertion $3^\circ$ follows from  (\ref{bilinenergyVm}) and (\ref{defprodscgeneral}). By \cite[Theorem 1.4.2]{St06}, the quotient of $H^1 (V)$ by the constant functions, equipped with the inner product $\langle\cdot ,\cdot\rangle$, is a Hilbert space. From here and from the fact that functions in $H^1 (V)$ are H\"older continuous (assertion $2^\circ$), one easily obtains assertion $4^\circ$. Assertion $5^\circ$ is an immediate consequence of assertion $2^\circ$ and the Arzel\`a-Ascoli theorem.
\end{proof}

The next result follows immediately from the definition of the energy form $W$.

\begin{lemma}\label{Lipschitz}
If $h\colon \R\to\R$ is a Lipschitz map with Lipschitz constant $L\geq0$, and if $u\in H^1(V)$, then $h\circ u\in H^1(V)$ and $W(h\circ u)\leq L^2W(u)$. 
\end{lemma}

\begin{corollary}\label{ausLipschitz}
The following assertions hold:
\begin{itemize}
\item[$1^\circ$] If $u\in H^1(V)$, then $u^+$, $u^-\in H^1(V)$.

\item[$2^\circ$] If $u,\,v\in H^1(V)$, then $u\vee v,\, u\wedge v\in H^1(V)$.

\item[$3^\circ$] If $A$ and $B$ are nonempty, disjoint, closed subsets of $V$, then there exists $u\in H^1(V)$, $u\colon V\to[0,1]$,  such that $u(A)=\{0\}$ and $u(B)=\{1\}$.
\end{itemize}
\end{corollary}

\begin{proof}
$1^\circ$ Pick $u\in H^1(V)$. The map $h\colon \R\to\R$, defined by $h(t)=\max\{t,\,0\}$, is clearly Lipschitz continuous. Thus, by Lemma \ref{Lipschitz}, we get that $u^+=h\circ u\in  H^1(V)$. Since $\,u^-=(-u)^+$, it follows that $u^-\in H^1(V)$, too.

\noindent $2^\circ$ The statement follows from $1^\circ$, since 
\[
u\vee v=u+(v-u)^+ \hbox{ and } u\wedge v=v-(v-u)^+.
\]

\noindent $3^\circ$ Let $A$ and $B$ be nonempty, disjoint, closed subsets of $V$. By Urysohn's Lemma there exists a continuous function $v\colon V\to [0,1]$ such that $v(A)=\{0\}$ and $v(B)=\{1\}$. In view of assertion $1^\circ$ of Theorem \ref{properties} there exists an element $w\in H^1(V)$ such that $\| v-w\|_{\sup}\leq\frac14$. Then $w\leq \frac14$ on $A$ and $w\geq \frac34$ on $B$. Define the cut-off function $h:\R\to\R$ by
\[
 h(s) = \begin{cases}
         0 & \text{if } s\leq 0 , \\[2mm]
         s & \text{if } s\in (0,1) , \\[2mm]
         1 & \text{if } s\geq 1 ,
        \end{cases}
\]
so that $h$ is Lipschitz continuous. Setting
\[
 u := h \circ \left( 2w-\frac12 \right) ,
\]
which belongs to $H^1 (V)$ by Lemma \ref{Lipschitz}, gives the desired element.   
\end{proof}

\begin{lemma}\label{stattpointwise}
Fix $y\in V$. Let $(u_n)$ be a sequence in $H^1(V)$ that converges in the $\|\cdot\|_{\sup}$ norm to $u\in H^1(V)$. If $W(u_n)\leq W(u)$ for every $n$, then   $(u_n)$  converges also in the $\|\cdot\|_{H^1}$ norm to $u$.
\end{lemma}
\begin{proof} By assumption, the sequence $(u_n)$ is bounded in the  $\|\cdot\|_{H^1}$ norm. Hence, by the reflexivity of the space $(H^1(V),\langle\cdot,\cdot\rangle_{H^1})$, this sequence has a subsequence $(u_{n_k})$ that converges in the weak topology to an element $\bar u\in H^1(V)$.  Assertion $5^\circ$ of Theorem \ref{properties}  implies that $(u_{n_k})$ converges in the $\|\cdot\|_{\sup}$ norm to $\bar u$, thus $\bar u=u$.

On the other hand, since $W$ is convex and continuous in the $\|\cdot\|_{H^1}$ norm, it is weakly lower semicontinuous. So, the inequalities $W(u_{n_k})\leq W(u)$, for all $k$, yield that $\displaystyle\lim_{k\to\infty}W(u_{n_k})=W(u)$. It follows that $\displaystyle\lim_{k\to\infty}\|u_{n_k}\|_{H^1}=\|u\|_{H^1}$. Since 
$(H^1(V),\langle\cdot,\cdot\rangle_{H^1})$ is uniformly convex, the weak convergence of $(u_{n_k})$ to $u$ implies now the convergence in the $\|\cdot\|_{H^1}$ norm of this subsequence to $u$.

The above argument shows in fact that every subsequence of $(u_n)$ has a subsequence that converges in the $\|\cdot\|_{H^1}$ norm to $u$. This  yields finally the convergence of $(u_n)$ in the $\|\cdot\|_{H^1}$ norm to $u$. 
\end{proof}

\begin{corollary}\label{ausstattpointwise}
Fix $y\in V$ and let $u\in H^1(V)$. For every $n\in\N^*$ define  the map $h_n\colon\R\to\R$ by
\[
h_n(t)=
\left\{
\begin{array}{ll}
t+\frac1n &\text{if } t<-\frac1n , \\[2mm]
0 &\text{if } -\frac1n\leq t\leq \frac1n , \\[2mm]
t-\frac1n & \text{if } t>\frac1n.
\end{array}
\right.
\]
Then $h_n\circ u\in H^1(V)$, for every $n\in\N^*$, and the sequence $(h_n\circ u)$ converges in the $\|\cdot\|_{H^1}$ norm to $u$.
\end{corollary}

\begin{proof}
Note that every $h_n$, $n\in\N^*$, is a Lipschitz map with Lipschitz constant $1$. Thus, Lemma \ref{Lipschitz} yields that $h_n\circ u\in H^1(V)$ and $W(h_n\circ u)\leq W(u)$, for every $n\in\N^*$. Furthermore, by definition of $h_n$, we get that $\|h_n\circ u-u\|_{\sup}\leq \frac1n$, for every $n\in\N^*$. Thus $(h_n\circ u)$ converges in the $\|\cdot\|_{\sup}$ norm to $u$. Lemma \ref{stattpointwise} yields the claim.
\end{proof}

\begin{proposition} \label{orthogonal}
The following assertions hold:
\begin{itemize}
\item[$1^\circ$] For every $u$, $v\in H^1(V)$ satisfying $|u|\wedge |v|=0$ one has that
\begin{align*}
 & \langle u,v\rangle =0 \quad \text{and} \\
 & W(u+v)=W(u)+W(v) .
\end{align*}

\item[$2^\circ$] If $u\in H^1(V)$, then $|u|\in H^1(V)$ and $W(u)=W(|u|)$.

\item[$3^\circ$] For every $u$, $v\in H^1(V)$
\[
 W(u\vee v) + W(u\wedge v) = W(u) + W(v) .
\]
\end{itemize}
\end{proposition}

\begin{proof}
$1^\circ$ We assume first that ${\rm supp}\, u \cap {\rm supp}\, v=\emptyset$. Set $\delta:={\rm dist}({\rm supp}\, u,{\rm supp}\, v)$. Then $\delta>0$. For every $m\in\N$ such that $\frac1{2^m}<\delta$, and every $x,\,y\in V_m$ with $x\underset{m}{\sim} y$ we then have $|x-y|=\frac1{2^m}<\delta$, thus $u(x)\cdot v(y)=0$. It follows that $\langle u,v\rangle_m=0$, hence $\displaystyle \langle u,v\rangle=\lim_{m\to\infty} \langle u,v\rangle_m= 0$.

In the general case, consider the maps $h_n$, $n\in\N^*$, defined in Corollary \ref{ausstattpointwise}, and put $u_n:=h_n\circ u$ and $v_n:=h_n\circ v$. By Corollary \ref{ausstattpointwise}, $(u_n)$ and $(v_n)$ converge in the $\|\cdot\|_{H^1}$ norm to $u$ and $v$, respectively. On the other hand, the definition of $h_n$ implies that 
\[
{\rm supp}\, u_n\cap{\rm supp}\, v_n=\emptyset, \ \forall\ n\in\N^*.
\]
So, from what we have proved at the beginning, it follows that $\langle u_n,v_n\rangle=0$, for every $n\in\N^*$. The continuity of the map $\langle \cdot,\cdot\rangle$ in  the $\|\cdot\|_{H^1}$ norm then implies that $\langle u,v\rangle=0$. From here we conclude that
$$W(u+v)=W(u)+\langle u,v\rangle+W(v)=W(u)+W(v).$$

\noindent $2^\circ$ Assertion $1^\circ$ of Corollary \ref{ausLipschitz} yields that $|u| = u^+ +u^-\in H^1(V)$.
Since  $u^+\cdot u^-=0$, we get, by assertion $1^\circ$, that
\[
W(u)=W( u^+ - u^-)=W(u^+)+W(u^-)=W( u^+ + u^-)=W(|u|).
\]

\noindent $3^\circ$ For every $u$, $v\in H^1 (V)$ we have, by assertion $1^\circ$, 
\begin{align*}
 W(u\vee v) + W(u\wedge v) & = W(u+ (v-u)^+) + W(v-(v-u)^+) \\
 & = W(u) + W(v) - \langle v-u , (v-u)^+\rangle + \langle (v-u)^+ , (v-u)^+\rangle \\
 & = W(u) + W(v) + \langle (v-u)^- , (v-u)^+ \rangle \\
 & = W(u) + W(v) . 
\end{align*}
This finishes the proof.
\end{proof}

We state now further properties of the energy $W$ that will be used in the sequel. These properties are in fact consequences of some basic inequalities concerning real numbers mentioned in the following lemma (whose straightforward proof is omitted).

\begin{lemma}\label{Ungleichungen}
Let $a_1,\,a_2,\,b_1,\,b_2,\,\alpha\in\R$ with $\alpha>0$. For $i\in\{1,\,2\}$ denote by 
\begin{align*}
A_i & :=\min\left\{\max\left\{a_i,\frac{a_i+b_i-\alpha}2\right\},\frac{a_i+b_i+\alpha}2\right\}, \\
B_i & :=\max\left\{\min\left\{b_i,\frac{a_i+b_i+\alpha}2\right\},\frac{a_i+b_i-\alpha}2\right\}, \\
C_i & :=\min\{|a_i|,\,b_i\}\cdot\sgn(a_i), \\
D_i & :=\max\{|a_i|,\,b_i\}.
\end{align*}
Then 
\begin{equation}\label{forcontractive1}
(A_1-A_2)^2+(B_1-B_2)^2\leq (a_1-a_2)^2+(b_1-b_2)^2,
\end{equation}
and, if $b_1\geq0$ and $b_2\geq0$,
\begin{equation}\label{fordomi1}
(C_1-C_2)^2+(D_1-D_2)^2\leq (a_1-a_2)^2+(b_1-b_2)^2.
\end{equation}
\end{lemma}

\begin{proposition}\label{propertiesW}
Let $u,\,v\in H^1(V)$ and let $\alpha$ be a positive real.  Denote by
\begin{equation}\label{Bezeichnung}
f:=\frac{u+v-\alpha}2,\ g:=\frac{u+v+\alpha}2.
\end{equation}
Then 
\begin{equation}\label{forcontractive2}
W((u\vee f)\wedge g)+W((v\wedge g)\vee f)\leq W(u)+W(v),
\end{equation}
and, if $v\geq0$,
\begin{equation}\label{fordomi2}
W\left((|u|\wedge v)\sgn(u)\right)+W(|u|\vee v)\leq W(u)+W(v).
\end{equation}
\end{proposition}
\begin{proof}
Let $m\in\N$ be arbitrary. According to (\ref{forcontractive1}), we then have for every $x$, $y\in V_m$ with $x\underset{m}{\sim} y$
\begin{multline*}
\left((u(x)\vee f(x)) \wedge g(x) - (u(y) \vee f(y)) \wedge g(y) \right)^2 + \\
+ \left((v(x) \wedge g(x)) \vee f(x) - (v(y) \wedge g(y) ) \vee f(y) \right)^2 \\
\leq (u(x)-u(y))^2+(v(x)-v(y))^2.
\end{multline*}
Thus
\[
W_m((u\vee f) \wedge g)+W_m((v \wedge g) \vee f) \leq W_m(u) + W_m(v), \ \forall\ m\in\N.
\]
Taking the limit $m\to\infty$, we obtain (\ref{forcontractive2}).

One can prove in a similar way that inequality (\ref{fordomi2}) follows from (\ref{fordomi1}).
\end{proof}

\section{Laplace operators with boundary conditions}

We retain the notations from the previous section, and consider a nonzero, finite Borel measure $\mu$ on $V$ with the property that every nonempty, relatively open subset of $V$ has strictly positive $\mu$-measure, and that the intrinsic boundary $V_0$ is a $\mu$-null set. Then, by \cite[Lemma 26.2 and Theorem 29.14]{Bau01}, $C_0(V)$ and $C(V)$ are densely and continuously embedded into $L^2_\mu (V)$. By Theorem \ref{properties}, $H^1_0(V)$ and $H^1(V)$ are then densely and continuously embedded in $(L^2_\mu (V),\|\cdot\|_{L^2_\mu})$, too. Note, by a direct calculation, that 
\[
 \langle u,v\rangle_{H^1_\mu} = \langle u,v\rangle_{L^2_\mu} + \langle u,v\rangle  \quad (u, \, v\in H^1 (V))
\]
defines an inner product on $H^1 (V)$ which is equivalent to the inner product $\langle\cdot ,\cdot\rangle_{H^1}$ from assertion $3^\circ$ of Theorem \ref{properties}. \\

Given nonnegative functions $B_1$, $\dots$, $B_N : \R \to [0,\infty ]$ and setting $B = (B_1,\dots ,B_N)$, we define the perturbed energy $W_B\colon L^2_\mu (V)\to[0,\infty]$ by 
\begin{equation}\label{defperturbedenergy}
 W_B (u) := \begin{cases}
W(u) + \sum_{i=1}^N B_i (u(p_i)) & \text{if } u\in H^1 (V) , \\[2mm]
\infty & \text{else.}
\end{cases}
\end{equation}

\begin{proposition} \label{ausallgemein2}
Let $B_1$, $\dots$, $B_N : \R \to [0,\infty ]$ be lower semicontinuous, and set $B = (B_1,\dots ,B_N)$. Then the perturbed energy $W_B$ is lower semicontinuous on $L^2_\mu (V)$. 
\end{proposition}

\begin{proof}
Let $(u_n)$ be a sequence in $L^2_\mu (V)$ which converges (in $L^2_\mu$) to some element $u\in L^2_\mu (V)$. We have to show that
\begin{equation} \label{eq.W_B}
W_B(u) \leq \liminf_{n\to\infty} W_B(u_n) .
\end{equation}
This inequality is trivially satisfied if the right-hand side equals $\infty$. So, after passing to a subsequence (denoted again by $(u_n)$), we may assume that $(W_B (u_n))_n$ is bounded. Since the functions $B_i$ are nonnegative, the sequence $(W(u_n))_n$ is thus bounded, too. Since $(u_n)$ is, as a convergent sequence, necessarily bounded in $L^2_\mu (V)$, we find that $(u_n)$ is bounded in $(H^1 (V),\|\cdot\|_{H^1_\mu})$. This space being reflexive and continuously embedded into $L^2_\mu (V)$, so that we can use a uniqueness of limit argument, we deduce that $(u_n)$ converges weakly in $H^1 (V)$ to $u$. Since $W$ is continuous on $H^1 (V)$ and convex,  it is lower semicontinuous with respect to the weak topology on $H^1 (V)$. Hence,
\[
W(u) \leq \liminf_{n\to\infty} W(u_n) .
\]
Now, using the compactness of the embedding of $(H^1 (V),\|\cdot\|_{H^1_\mu})$ into $(C(V),\|\cdot\|_{\sup})$ (assertion $5^\circ$ of Theorem \ref{properties}), we see that $(u_n)$ converges uniformly to $u$. In particular, $(u_n)$ converges pointwise everywhere. 
Moreover, since the functions $B_i$ are lower semicontinuous by assumption, we obtain
\[
 \sum_{i=1}^N B_i (u(p_i)) \leq \liminf_{n\to\infty} \sum_{i=1}^N B_i (u_n(p_i)) .
\]
Taking the preceding two estimates together, we obtain (\ref{eq.W_B}), and thus $W_B$ is lower semicontinuous on $L^2_\mu (V)$.
\end{proof}

Recall from \cite[Th\'eor\`eme 3.2]{Br73} that every proper, convex, lower semicontinuous energy $\tilde{W} : H\to [0,\infty ]$ on a real Hilbert space $H$ generates a strongly continuous semigroup of (possibly nonlinear) contractions on the closure of the effective domain ${\rm dom}\, \tilde{W} := \{\tilde{W} <\infty\}$. More precisely, if we define the {\em subgradient} of $\tilde{W}$ at a point $u\in {\rm dom}\,\tilde{W}$ by
\[
 \partial\tilde{W} (u) := \{ f\in H \mid \tilde{W} (u+v) - \tilde{W} (u) \geq \langle f,v\rangle, \ \forall v\in H \} ,  
\]
then, for every initial value $u_0$ in the closure of the effective domain ${\rm dom}\,\tilde{W}$, there exists a unique function $u\in C(\R_+ ;H ) \cap H^1_{loc} ((0,\infty );H)$ such that 
\begin{equation} \label{cauchy}
 \begin{cases}
  \dot u + \partial\tilde{W} (u) \ni 0 & \text{almost everywhere on } (0,\infty ) , \\[2mm]
  u(0) = u_0 . & 
 \end{cases}
\end{equation}
Defining $S(t)u_0 := u(t)$, we obtain the strongly continuous semigroup of contractions mentioned above. We say that a strongly continuous semigroup $S$ on an $L^2$ space is 
\begin{itemize}
 \item[(i)] {\em positive} if $S(t)u\geq 0$ for every $t\geq 0$ and every $u\geq 0$,
 \item[(ii)] {\em order preserving} if  $S(t)u\leq S(t)v$ for every $t\geq 0$ and every $u$, $v\in L^2$ with $u\leq v$, 
 \item[(iii)] {\em $L^\infty$-contractive} if $\| S(t) u - S(t)v\|_{\sup} \leq \| u-v\|_{\sup}$ for every $t\geq 0$ and  every $u$, $v\in L^2$.
\end{itemize}
Given two strongly continuous semigroups $S_1$, $S_2$ on $L^2$, we say that
\begin{itemize}
 \item[(iv)] $S_1$ {\em is dominated by } $S_2$ (and we write $S_1 \preccurlyeq S_2$) if $S_2$ is order preserving and $| S_1(t)u| \leq S_2 (t)v$ for all $t\geq 0$ and all $u$, $v\in L^2$ such that $|u|\leq v$, and
 \item[(v)] $S_1$ {\em is totally dominated by } $S_2$ (and we write $S_1 \stackrel{t}{\preccurlyeq}S_2$) if $S_1 \preccurlyeq S_2$ and $S_1 \preccurlyeq -S_2 (-\cdot )$.
\end{itemize}
Note that $S_1$ is totally dominated by $S_2$ if and only if $S_2$ is order preserving and for all $t\geq 0$ and all $u$, $v\in L^2$
\[
 |S_1 (t)u| \leq \begin{cases}
                  S_2 (t)v & \text{if } |u| \leq v , \\[2mm]
                  -S_2 (t)(v) & \text{if } v \leq -|u| .
                 \end{cases}
\]
Total domination of $S_1$ by $S_2$ corresponds to a domination of $S_1$ by $S_2$ from above, that is, in the sense of definition (iv), and to a domination of $S_1$ by $S_2$ from below. Note that if the semigroup $S_2$ is antisymmetric with respect to the origin in the sense that $S_2 (-\cdot ) = -S_2 (\cdot )$, then the domination $S_1\preccurlyeq S_2$ is equivalent to total domination $S_1\stackrel{t}{\preccurlyeq} S_2$. Note also that linear semigroups are antisymmetric. 

If the semigroup $S$ (or, the semigroups $S_1$ and $S_2$) are generated by convex, lower semicontinuous energies on $L^2$, then the above properties can be fully characterised via the generating energies; see \cite{By96,Br73,BrPa70,CiGr03}, but also \cite{Ou04} in the case of {\em quadratic} energies and linear semigroups. In particular, one obtains Theorem \ref{thm.qualitative} below for which we introduce the following notions.
A function $B: \R\to [0,\infty ]$ is called {\em bi-monotone} if it is decreasing on $\R_-$ and increasing on $\R_+$, and it is called {\em normalised} if $B(0)=0$. Note that a bi-monotone function attains its minimum at $0$, and that  normalised functions are proper. 

\begin{theorem} \label{thm.qualitative}
Let $B_1$, $\dots$, $B_N : \R\to [0,\infty ]$ be normalised functions, and set $B= (B_1,\dots , B_N)$. Assume that the energy $W_B$ is lower semicontinuous and convex, and let $S_B = (S_B(t))_{t\geq 0}$ be the semigroup generated by $W_B$. Then:
\begin{itemize}
 \item[$1^\circ$] The semigroup $S_B$ is positive and order preserving.
 \item[$2^\circ$] If $B$ is convex, then the semigroup $S_B$ is $L^\infty$-contractive.
\end{itemize}
Moreover, let $\hat{B}_1$, $\dots$, $\hat{B}_N: \R\to [0,\infty ]$ be a second set of normalised functions, $\hat{B} = (\hat{B}_1,\dots , \hat{B}_N)$. Assume also that $W_{\hat{B}}$ is lower semicontinuous and convex, and let $S_{\hat{B}} = (S_{\hat{B}} (t))_{t\geq 0}$ be the semigroup generated by $W_{\hat{B}}$. Then:
\begin{itemize}
 \item[$3^\circ$] If $s\mapsto \hat{B}(s) - B(|s|)$ is bi-monotone (with the interpretation $\infty - \infty = \infty$), then $S_{\hat{B}}\preccurlyeq S_B$.
 \item[$4^\circ$] If $s\mapsto \hat{B}(s) - B(|s|)$ and $s\mapsto \hat{B}(s) - B(-|s|)$ are bi-monotone (with the same interpretation as above), then $S_{\hat{B}}\stackrel{t}{\preccurlyeq} S_B$
\end{itemize}
\end{theorem}

\begin{proof}
 By the assumptions  on $W_B$, this energy  really generates a strongly continuous semigroup of contractions on the closure of its effective domain. However, note that $H^1_0 (V) \subseteq {\rm dom}\, W_B$, and recall that $H^1_0 (V)$ is dense in $L^2_\mu (V)$. Hence, the semigroup $S_B$ is defined on the entire space $L^2_\mu (V)$. Let us turn to prove the required qualitative properties.

{\em Order preservingness.} By \cite[Corollaire~2.2]{By96}, the semigroup $S_B$ is order preserving if and only if, for all $u$, $v\in L^2_\mu (V)$,
\begin{equation}\label{order}
W_B(u\wedge v)+W_B(u\vee v)\leq W_B(u)+W_B(v).
\end{equation}
The above inequality clearly holds if $u$ or $v$ do not belong to the effective domain ${\rm dom}\, W_B$. Thus we may assume that $u$ and $v$ both lie in ${\rm dom}\, W_B$. In particular, $u$ and $v$ are in $H^1(V)$. For every $i\in\{1,\dots,N\}$ we clearly have
\[
B_i ((u\wedge v)(p_i)) + B_i ((u\vee v)(p_i)) = B_i (u(p_i)) + B_i (v(p_i)) . 
\]
Combining this with assertion $3^\circ$ of Proposition \ref{orthogonal}, we obtain (\ref{order}), even with equality.

{\em Positivity.} By \cite[Th\'eor\`eme 1.9]{By96} applied to the convex set $L^2_\mu (V)^+$, the semigroup $S_B$ is positive if and only if, for every $u\in L^2_\mu (V)$,
\begin{equation} \label{positive}
 W_B (u^+) \leq W_B(u) .
\end{equation}
This inequality, however, follows directly from (\ref{order}) by putting $v=0$ and noting that $W_B (u^+ ) = W_B (u\vee 0)$, $W_B(u \wedge 0) \geq 0$ and $W_B (0)=0$. 

Alternatively, one may argue that the origin is a global minimum of $W_B$ and thus an equilibrium point in the sense that $S(t)0 = 0$ for every $t\geq 0$. Now the positivity follows from the first part, namely that $S_B$ is order preserving.

{\em $L^\infty$-contractivity.} We assume now that $B$ is convex and prove that $S_B$ is $L^\infty$-contractive. For this we use \cite[Lemma 3.3 and Theorem 3.6]{CiGr03}, that is, we have to show that for every $u$, $v\in L^2_\mu (V)$ and for every $\alpha>0$ one has that 
\begin{equation}\label{contractive}
W_B((u\vee f)\wedge g)+W_B((v\wedge g)\vee f)\leq W_B(u)+W_B(v),
\end{equation}
where $f$ and $g$ are defined in (\ref{Bezeichnung}). As above, we may assume that $u$ and $v$ lie in ${\rm dom}\, W_B\subseteq H^1 (V)$.  It can be easily seen, using the convexity of the functions $B_i$, that for every $i\in\{1,\dots,N\}$,
\begin{multline*}
B_i \left(\min\{\max\{u(p_i),f(p_i)\},g(p_i)\}\right) + B_i \left(\max\{\min\{v(p_i),g(p_i)\},f(p_i)\}\right) \\
\leq B_i(u(p_i)) + B_i (v(p_i)).
\end{multline*}
Hence, using (\ref{forcontractive2}), we get (\ref{contractive}).

{\em Domination.} Assume that $s\mapsto \hat{B}(s) - B(|s|)$ is bi-monotone. In order to show that $S_{\hat{B}}\preccurlyeq S_B$, we have to check, according to \cite[Th\'eor\`eme~3.3]{By96}, that for all $u$, $v\in L^2_\mu (V)$ with $v\geq 0$ 
\begin{equation}\label{domi}
W_{\hat{B}}\left((|u|\wedge v)\sgn(u)\right)+W_B(|u|\vee v) \leq W_{\hat{B}}(u) + W_B(v).
\end{equation}
As above, we may assume that $u\in {\rm dom}\, W_{\hat B}$ and $v\in {\rm dom}\, W_B$; in particular, $u$, $v\in H^1 (V)$. Since $\hat{B}_i-B_i$ is bi-monotone, it can be easily seen that, for every $i\in\{1,\dots,N\}$, 
\[
\hat B_i\left(((|u|\wedge v)\sgn(u)) (p_i)\right) + B_i \left( (|u|\vee v)(p_i)\right) \leq \hat B_i (u(p_i)) + B_i (v(p_i)).
\]
So, using (\ref{fordomi2}), we obtain (\ref{domi}).

{\em Total domination.} This point follows from the preceding point and the observation that the semigroup $-S_B (-\cdot )$ is generated by the convex, lower semicontinuous energy $W_B (-\cdot )$. 
\end{proof}

At the end of this section we discuss the type of operators which we obtain as subgradients of energies of the form $W_B$. First, we introduce the Laplace operator $\Delta_\mu$ on $L^2_\mu (V)$ by
\begin{align*}
 {\rm dom}\, \Delta_\mu & := \{ u\in H^1 (V) | \exists f\in L^2_\mu (V) \forall v\in H^1_0 (V) \, \langle u,v\rangle = \langle f,v\rangle_{L^2_\mu} \} , \\
 -\Delta_\mu u & := f .
\end{align*}
Note that this operator as well as the space on which it acts both depend on the Borel measure $\mu$. In the special case when $\mu = \mathcal H^d / \mathcal H^d (V)$ is the normalised $d$-dimensional Hausdorff measure, where $d = \log N / \log 2$ is the Hausdorff dimension of the Sierpinski gasket, the operator $\Delta_\mu$ is the Laplace operator introduced by Kigami \cite{Ki89,Ki93}; we denote it simply by $\Delta$. Kigami also defined the normal derivate $\frac{\partial u}{\partial\nu}$ of a function $u\in {\rm dom}\, \Delta$ on the intrinsic boundary; see \cite{Ki89}. In each point of the intrinsic boundary it can be defined as a limit of a sort of difference quotients of the function $u$ on the subgraphs $V_m$. It turns out that $u\in {\rm dom}\, \Delta$ is a solution of the boundary value problem
\begin{align*}
 -\Delta u & = f \text{ on } V , \\
 \frac{\partial u}{\partial\nu} & = g \text{ on } V_0 ,
\end{align*}
if and only if 
\[
 \langle u , v\rangle = \langle f,v\rangle_{L^2} + \sum_{i=1}^N g(p_i) v(p_i ) \text{ for every } v\in H^1 (V) ; 
\]
observe that the test functions $v$ now run through $H^1 (V)$ while in the definition of $\Delta$ they only vary in $H^1_0 (V)$. Let now $B_1$, $\dots$, $B_N :\R\to [0,\infty ]$ be continuously differentiable functions and put $B= (B_1, \dots , B_N )$. Then, for every $u$, $v\in H^1 (V)$
\[
 \lim_{\lambda\to 0} \frac{W_B (u+\lambda v) - W_B (u)}{\lambda} = \langle u,v\rangle + \sum_{i=1}^N B_i' (u(p_i)) v(p_i ) .
\]
As a consequence, $f\in\partial W_B (u)$ if and only if
\begin{align*}
 -\Delta u & = f \text{ on } V , \\
 \frac{\partial u}{\partial\nu} (p_i) + B_i' (u(p_i)) & = 0 \text{ for every } i\in \{ 1, \dots , N\} ,
\end{align*}
that is, $u\in {\rm dom}\,\Delta$ is a solution of a Poisson equation with Robin boundary conditions. The Cauchy problem \eqref{cauchy} associated with a lower semicontinuous, convex energy $W_B$, now for a general Borel measure $\mu$ and general lower semicontinuous, convex functions $B_i$ can thus be rewritten formally in the form
\begin{align*}
 \partial_t u - \Delta_\mu u & = 0 \text{ in } \R_+ \times V , \\
  \frac{\partial u}{\partial\nu} (p_i) + \partial B_i (u(p_i)) & \ni 0 \text{ for every } i\in \{ 1, \dots , N\} .
\end{align*}
This is a diffusion type equation in the space $L^2_\mu (V)$ associated with the Laplace type operator $\Delta_\mu$ and with general Robin type boundary conditions. 
 
\section{What is in between?}

Within the class of energies $W_B$ associated with lower semicontinuous, bi-mo\-no\-tone and normalised functions $B_i$ and with respect to the order given by comparison pointwise everywhere the energy $W$ (extended to  $L^2_\mu (V)$ by $\infty$) is the smallest energy, while $W_\infty := W_{(\infty , \dots ,\infty)}$ associated with the functions 
\[
 B^\infty_i (u) = \begin{cases}
            \infty & \text{if } u\not= 0 , \\[2mm]
            0 & \text{if } u=0 ,
           \end{cases}
\]
(which are not precisely equal to $\infty$ in order to achieve normalization) is the largest energy. Note that $W$ and $W_\infty$ are lower semicontinuous and convex. Denote by $S$ and $S_\infty$ the semigroups generated by $W$ and $W_\infty$, respectively. For any $N$-tuple $B= (B_1 , \dots , B_N)$ of lower semicontinuous, bi-monotone and normalised functions $B_i : \R\to [0,\infty ]$ one has
\[
 W \leq W_B \leq W_\infty \text{ pointwise everywhere.}
\]
If $W_B$ is convex, then Theorem \ref{thm.qualitative} yields that the associated semigroups $S$, $S_B$ and $S_\infty$ satisfy the opposite relation in the sense of total domination. More precisely one obtains the following result.

\begin{corollary} \label{aus.thm.qualitative}
Let $B_1$, $\dots$, $B_N: \R\to [0,\infty ]$ be lower semicontinuous, bi-monotone and normalised functions, $B= (B_1,\dots , B_N)$. Assume that the energy $W_B$ is convex, and let $S_B = (S_B(t))_{t\geq 0}$ be the semigroup generated by $W_B$. Then $S_B$ is positive, order preserving and
\[
 S_\infty \stackrel{t}\preccurlyeq S_B \stackrel{t}\preccurlyeq S .
\]
\end{corollary}
\begin{proof}
Proposition \ref{ausallgemein2} implies that $W_B$ is lower semicontinuous. The assertion follows now from Theorem \ref{thm.qualitative}.
\end{proof}

In this section we discuss a partial converse of Corollary \ref{aus.thm.qualitative}. More exactly,
the question arises whether there are more semigroups sandwiched between $S_\infty$ and $S$. In this section we answer this question in the negative, at least if we  restrict ourselves to {\em local} semigroups. Loosely speaking, an operator $A$ on a function space is {\em local} if the support of $Au$ is contained in the support of $u$, for all $u$ in the domain of the operator. In the case of the subgradient of a convex, lower semicontinuous functional $\tilde{W}$ this can be written as
\[
 {\rm supp}\, f \subseteq {\rm supp}\, u \text{ whenever } f\in\partial\tilde{W} (u) .
\]
This property can be translated back to the functional $\tilde{W}$ itself. Actually, we say that the functional $\tilde{W} : L^2_\mu (V)\to [0,\infty ]$ is {\em local} if, for every $u$, $v\in L^2_\mu (V)$,
\[
 \tilde{W} (u+v) = \tilde{W} (u) + \tilde{W} (v) \text{ whenever } |u|\wedge |v| = 0 .
\]
For obvious reasons, such functionals are sometimes also called {\em additive}. 

\begin{lemma}[Locality of the energies]
For every $N$-tuple $B=(B_1,\dots , B_N)$ of lower semicontinuous, bi-monotone, normalised functions $B_i:\R\to [0,\infty ]$ the energy $W_B$ is local.
\end{lemma}

\begin{proof}
Let $u$, $v\in L^2_\mu (V)$ be such that $|u|\wedge |v| = 0$. Assume first that $W_B (u) = \infty$ or $W_B(v) =\infty$. If $W_B (u) = \infty$, then $u\not\in H^1 (V)$ or, if $u\in H^1 (V)$, $B_i (u(p_i)) = \infty$ for some $i\in \{1,\dots ,N\}$, and similarly for $v$ if $W_B (v) =\infty$.  

If $u\not\in H^1 (V)$ or $v\not\in H^1 (V)$, then $|u|\wedge |v|=0$ implies $u+v\not\in H^1 (V)$. In this case, $W_B (u+v) = W_B (u) +W_B(v) (=\infty)$. 

If $u$, $v\in H^1 (V)$ and $B_i (u(p_i)) = \infty$ for some $i\in \{1,\dots ,N\}$, then $u(p_i)\not= 0$ and $|u|\wedge |v| = 0$ implies $v(p_i ) = 0$. Therefore, $B_i (u(p_i)+v(p_i)) = B_i (u(p_i)) = \infty$, and therefore $W_B (u+v) = W_B (u) +W_B(v) (=\infty )$. By symmetry, this equality also holds if $B_i (v (p_i))=\infty$ for some $i\in \{1,\dots ,N\}$.  

So assume finally that $u$ and $v$ both belong to the effective domain ${\rm dom}\, W_B$. In particular, $u$, $v\in H^1 (V) \subseteq C(V)$. Since $u\cdot v=0$, and since the $B_i$ are normalised, we have for every $i\in\{1,\dots,N\}$ that
\[
B_i (u(p_i) + v(p_i)) = B_i (u(p_i)) + B_i (v(p_i)) ,
\]
so the statement in this case follows from assertion $1^\circ$ of Proposition \ref{orthogonal}.
\end{proof}

\begin{theorem}
Let $\tilde{S}$ be a strongly continuous, order preserving semigroup of contractions on $L^2_\mu (V)$. Assume that $\tilde{S}$ is generated by a convex, proper, lower semicontinuous, local function $\tilde{W} : L^2_\mu (V) \to [0,\infty ]$ such that $\tilde{W} (0)=0$. Assume that
\[
 S_\infty \stackrel{t}{\preccurlyeq} \tilde{S} \stackrel{t}{\preccurlyeq} S .
\]
Then there exist lower semicontinuous, bi-monotone, normalised functions $B_1$, $\dots$, $B_N : \R\to [0,+\infty ]$ such that $\tilde{W} = W_B$, that is,
\[
 \tilde{W} (u) = \begin{cases}
                  W (u) + \sum_{i=1}^N B_i (u(p_i)) & \text{if } u\in H^1 (V) , \\[2mm]
                  \infty & \text{else.}
                 \end{cases}
\]
\end{theorem}

\begin{proof}
Throughout the proof, the space $H^1 (V)$ is endowed with the topology induced by the $\|\cdot\|_{H^1_\mu}$ norm.
Consider the function $\Psi := \tilde{W} -W$ on the space $H^1 (V)$. This function is well defined since the effective domain of $W$ is $H^1 (V)$. Moreover, since $W$ is continuous on $H^1 (V)$, the function $\Psi$ is lower semicontinuous on $H^1 (V)$. 

{\em 1st step.} We show, using the domination $\tilde{S}\preccurlyeq S$, that the function $\Psi$ is monotone on the positive cone of $H^1 (V)$ in the sense that $0\leq v\leq u$ implies $\Psi (v) \leq \Psi (u)$. In fact, by \cite[Th\'eor\`eme~ 3.3]{By96}, for every $u$, $v\in L^2_\mu (V)$ with $v\geq 0$,
\begin{equation}\label{ausBy}
 \tilde{W} ((|u|\wedge v)\sgn (u)) + W (|u|\vee v) \leq \tilde{W} (u) + W(v) .
\end{equation}
Choosing $u$, $v\in H^1 (V)$ such that $0\leq v\leq u$, this inequality implies
\[
 \tilde{W} (v) + W(u) \leq \tilde{W} (u) + W(v) ,
\]
and from here follows the monotonicity of $\Psi$. 

Let us remark, in addition, that the above characterization of domination, together with the fact that $\tilde{S}$ is order preserving, implies ${\rm dom}\, \tilde{W}\subseteq H^1 (V)$. Since the semigroup $\tilde{S}$ is assumed to be order preserving, and by \cite[Corollaire~2.2]{By96}, we have for every $u$, $v\in L^2_\mu (V)$,
\begin{equation} \label{order2}
 \tilde{W} (u\vee v) + \tilde{W} (u\wedge v) \leq \tilde{W} (u) + \tilde{W} (v).
\end{equation}
Applying successively the inequality \eqref{ausBy} with $u$ replaced respectively by $u^+$ and $-u^-$ and with $v=0$, and then the inequality \eqref{order2} with $v=0$, we obtain that for every $u\in L^2_\mu (V)$
\begin{align*}
 W(u^+) + W(u^-) & \leq \tilde{W} (u^+) + \tilde{W} (-u^-) \leq \tilde{W}(u) .
\end{align*}
Hence, if $\tilde{W} (u)$ is finite, then $u^+$, $u^-\in H^1 (V)$, and therefore $u = u^+ - u^-\in H^1 (V)$.

{\em 2nd step.} In a similar way, using now the domination $S_\infty \preccurlyeq \tilde{S}$, one shows that the function $W_\infty - \tilde{W}$ is monotone on the positive cone of $H^1_0 (V)$. Since $W_\infty$ and $W$ coincide on $H^1_0 (V)$, this means that $-\Psi$ is monotone on the positive cone of $H^1_0 (V)$. As a consequence, $\Psi$ is constant on the positive cone of $H^1_0 (V)$. The normalization $\tilde{W} (0)=0$ then implies $\Psi =0$ on the positive cone of $H^1_0 (V)$. 

{\em 3rd step.} We now show that for every nonnegative $u$, $v\in H^1 (V)$ the implication
\[
 u=v \text{ on the intrinsic boundary } V_0 \quad \Rightarrow \quad \Psi (u) = \Psi (v) 
\]
holds. Combining the inequality \eqref{order2} with assertion $3^\circ$ of Proposition \ref{orthogonal}, we obtain, for every $u$, $v\in H^1 (V)$,
\begin{equation} \label{eq.psi}
 \Psi (u\vee v) + \Psi (u\wedge v) \leq \Psi (u) + \Psi (v) .
\end{equation}
Now let $u$, $v\in H^1 (V)$ be nonnegative, and assume that $u = v$ on the intrinsic boundary $V_0$. 

Assume first that $u\leq v$ and that the support of $v-u\in H^1_0 (V)$ is contained in $V\setminus V_0$. Then there exists, according to assertion $3^\circ$ of Corollary \ref{ausLipschitz}, a nonnegative $w\in H^1_0 (V)$ such that $v= u\vee w$. By the second step, $\Psi (w) = \Psi (u\wedge w) = 0$, and hence the inequality \eqref{eq.psi} implies $\Psi (v) \leq \Psi (u)$. By monotonicity, the converse inequality holds, too, and thus $\Psi (u) = \Psi (v)$. 

Assume next merely that $u\leq v$. Then, involving Corollary \ref{ausstattpointwise}, we can approximate $v$ in the space $H^1 (V)$ by a sequence of functions $(v_n)$ such that $u\leq v_n$, and such that the support of $v_n -u$ is contained in $V\setminus V_0$. The preceding argument and the lower semicontinuity of $\Psi$ then give $\Psi (v) \leq \liminf_{n\to\infty} \Psi (v_n) = \Psi (u)$, and again the monotonicity of $\Psi$ yields the equality $\Psi (u) = \Psi (v)$. 

Finally, we only assume that $u=v$ on $V_0$. Since $u\leq u\vee v$ and $v\leq u\vee v$, and since the functions $u$, $v$ and $u\vee v$ coincide on the intrinsic boundary $V_0$, the preceding argument yields $\Psi (u) = \Psi (u\vee v) = \Psi (v)$, and we have proved the claim of this step. 

{\em 4th step.} Fix $i\in \{ 1,\dots ,N\}$. Let $w_i : V_1 \to \R$ be defined by
\[
 w_i (x) = \begin{cases}
              1 & \text{if } x=p_i , \\[2mm]
              0 & \text{else,}
             \end{cases}
\]
and let $u_i\in H^1 (V)$ be the harmonic extension of $w_i$ (see \cite[Section 3]{Br13}). By the maximum principle \cite[Corollary 3.9]{Br13}, $u_i \geq 0$. Moreover, the special form of the functions $w_i$ implies $|u_i|\wedge |u_j| = 0$ for $i\not= j$. Define $B_i^+ : [0,\infty ) \to [0,\infty ]$ by
\[
 B_i^+ (s) := \Psi (s\, u_i) .
\]
By the lower semicontinuity and monotonicity of $\Psi$, $B_i$ is lower semicontinuous and increasing, and since $\Psi (0) = 0$, we also have $B_i^+ (0)=0$. Now it follows easily from Step 3 and the locality assumption that, for every nonnegative $u\in H^1 (V)$,
\begin{align*}
\Psi (u) & = \Psi \left(\sum_{i=1}^N u(p_i) \, u_i \right) \\
& = \sum_{i=1}^N \Psi (u(p_i) \, u_i ) \\
& = \sum_{i=1}^N B_i^+ (u(p_i))  
\end{align*}
and hence
\[
 \tilde{W} (u) = W(u) + \sum_{i=1}^N B_i^+ (u(p_i)). 
\]

Repeating the Steps 1 to 4 for $\Psi$ restricted to the negative cone of $H^1 (V)$, using now the domination $S_\infty \preccurlyeq -\tilde{S} (-\cdot )$ and assertion $2^\circ$ of Proposition  \ref{orthogonal}, we obtain the existence of lower semicontinuous, decreasing functions $B_i^- : (-\infty ,0] \to [0,\infty ]$ such that $B_i^- (0) = 0$ and
\[
 \tilde{W} (u) = W(u) + \sum_{i=1}^N B_i^- (u(p_i)) 
\] 
for every  $u\in H^1 (V)$ with $u\leq0$. Define $B_i : \R \to [0,\infty ]$ by 
\[
 B_i (s) := \begin{cases}
             B_i^+ (s) & \text{if } s\geq 0 , \\[2mm]
             B_i^- (s) & \text{if } s\leq 0 .
            \end{cases}
\]
Then $B_i$ is lower semicontinuous, bi-monotone and normalised. Moreover, by locality of $\tilde{W}$ and $W$, for every $u\in H^1 (V)$,
\begin{align*}
 \tilde{W} (u) & = \tilde{W} (u^+) + \tilde{W} (-u^-) \\
 & = W (u^+) + \sum_{i=1}^N B_i^+ (u^+(p_i)) + W(-u^-) + \sum_{i=1}^N B_i^- (-u^-(p_i)) \\
 & = W(u) + \sum_{i=1}^N B_i (u(p_i)) .  
\end{align*}
The theorem is fully proved. 
\end{proof}

\nocite{St06}
\nocite{StWo04}
\nocite{HePeSt04}
\nocite{MeStTe04}
\nocite{Ou04}
\nocite{ChWa12}
\nocite{CiGr03}
\nocite{IoRoTe12}
\nocite{Br13}
\nocite{Ki89,Ki93}
\nocite{FuSh92}

\providecommand{\bysame}{\leavevmode\hbox to3em{\hrulefill}\thinspace}

\bibliographystyle{amsplain}

\begin{thebibliography}{10}

\bibitem{ArWa03b}
W.~Arendt and M.~Warma, \emph{Dirichlet and {N}eumann boundary
  conditions: {W}hat is in between?}, J. Evol. Equ. \textbf{3} (2003), no.~1,
  119--135, Dedicated to Philippe B{\'e}nilan.

\bibitem{By96}
L.~Barth\'elemy, \emph{Invariance d'un convex ferm\'e par un semi-groupe
  associ\'e \`a une forme non-lin\'eaire}, Abst. Appl. Anal. \textbf{1} (1996),
  237--262.

\bibitem{Bau01}
H.~Bauer, \emph{Measure and integration theory}, de Gruyter Studies in
  Mathematics, vol.~26, Walter de Gruyter \& Co., Berlin, 2001, Translated from
  the German by Robert B. Burckel.

\bibitem{BePi79a}
Ph.~B{\'e}nilan and C.~Picard, \emph{Quelques aspects non lin\'eaires
  du principe du maximum}, S\'eminaire de {T}h\'eorie du {P}otentiel, {N}o. 4
  ({P}aris, 1977/1978), Lecture Notes in Math., vol. 713, Springer, Berlin,
  1979, pp.~1--37.

\bibitem{Br13}
B.~E. Breckner, \emph{Real-valued functions of finite energy on the
  {S}ierpinski gasket}, Mathematica \textbf{55(78)} (2013), no.~2, 142--158.

\bibitem{Br73}
H.~Brezis, \emph{Op\'erateurs maximaux monotones et semi-groupes de
  contractions dans les espaces de {H}ilbert}, North Holland Mathematics
  Studies, vol.~5, North-Holland, Amsterdam, London, 1973.

\bibitem{BrPa70}
H.~Br\'ezis and A.~Pazy, \emph{Semigroups of nonlinear contractions on convex
  sets}, J. Funct. Anal. \textbf{6} (1970), 237--280.

\bibitem{ChWa12}
R. Chill and M. Warma, \emph{Dirichlet and {N}eumann boundary
  conditions for the {$p$}-{L}aplace operator: what is in between?}, Proc. Roy.
  Soc. Edinburgh Sect. A \textbf{142} (2012), no.~5, 975--1002. 

\bibitem{CiGr03}
F.~Cipriani and G.~Grillo, \emph{Nonlinear {M}arkov semigroups, nonlinear
  {D}irichlet forms and application to minimal surfaces}, J. reine angew. Math.
  \textbf{562} (2003), 201--235.

\bibitem{FuSh92}
M.~Fukushima and T.~Shima, \emph{On a spectral analysis for the {S}ierpi\'nski
  gasket}, Potential Anal. \textbf{1} (1992), no.~1, 1--35.

\bibitem{HePeSt04}
P.~E. Herman, R. Peirone, and R.~S. Strichartz, \emph{{$p$}-energy
  and {$p$}-harmonic functions on {S}ierpinski gasket type fractals}, Potential
  Anal. \textbf{20} (2004), no.~2, 125--148.

\bibitem{IoRoTe12}
M. Ionescu, L.~G. Rogers, and A. Teplyaev, \emph{Derivations and
  {D}irichlet forms on fractals}, J. Funct. Anal. \textbf{263} (2012), no.~8,
  2141--2169.

\bibitem{Ki89}
Jun Kigami, \emph{A harmonic calculus on the {S}ierpi\'nski spaces}, Japan J.
  Appl. Math. \textbf{6} (1989), no.~2, 259--290.

\bibitem{Ki93}
Jun Kigami, \emph{Harmonic calculus on p.c.f.\ self-similar sets}, Trans. Amer.
  Math. Soc. \textbf{335} (1993), no.~2, 721--755. 

\bibitem{MeStTe04}
R. Meyers, R.~S. Strichartz, and A. Teplyaev, \emph{Dirichlet
  forms on the {S}ierpi\'nski gasket}, Pacific J. Math. \textbf{217} (2004),
  no.~1, 149--174.

\bibitem{Ms97}
U. Mosco, \emph{Variational fractals}, Ann. Scuola Norm. Sup. Pisa Cl.
  Sci. (4) \textbf{25} (1997), no.~3-4, 683--712 (1998), Dedicated to Ennio De
  Giorgi.

\bibitem{Ms98b}
U. Mosco, \emph{Dirichlet forms and self-similarity}, New directions in
  {D}irichlet forms, AMS/IP Stud. Adv. Math., vol.~8, Amer. Math. Soc.,
  Providence, RI, 1998, pp.~117--155.

\bibitem{Ms98}
U. Mosco, \emph{Lagrangian metrics on fractals}, Recent advances in partial
  differential equations, {V}enice 1996, Proc. Sympos. Appl. Math., vol.~54,
  Amer. Math. Soc., Providence, RI, 1998, pp.~301--323.

\bibitem{Ou04}
E.~M. Ouhabaz, \emph{Analysis of {H}eat {E}quations on {D}omains}, London
  Mathematical Society Monographs, vol.~30, Princeton University Press,
  Princeton, 2004.

\bibitem{St06}
R.~S. Strichartz, \emph{Differential equations on fractals}, Princeton
  University Press, Princeton, NJ, 2006, A tutorial.

\bibitem{StWo04}
R.~S. Strichartz and C. Wong, \emph{The {$p$}-{L}aplacian on the
  {S}ierpinski gasket}, Nonlinearity \textbf{17} (2004), no.~2, 595--616.

\end{thebibliography}
  
\def\cprime{$'$} 
  \def\ocirc#1{\ifmmode\setbox0=\hbox{$#1$}\dimen0=\ht0 \advance\dimen0
  by1pt\rlap{\hbox to\wd0{\hss\raise\dimen0
  \hbox{\hskip.2em$\scriptscriptstyle\circ$}\hss}}#1\else {\accent"17 #1}\fi}
\providecommand{\bysame}{\leavevmode\hbox to3em{\hrulefill}\thinspace}
\providecommand{\MR}{\relax\ifhmode\unskip\space\fi MR }
\providecommand{\MRhref}[2]{%
  \href{http://www.ams.org/mathscinet-getitem?mr=#1}{#2}
}
\providecommand{\href}[2]{#2}

\end{document}